\documentclass[11pt, twoside]{amsart}

\usepackage{amsmath,amssymb,amsthm,enumitem,stmaryrd,mathrsfs}

\usepackage{thm-restate,hyperref}

\usepackage[pdftex]{graphicx}

\usepackage[top=1.5in, bottom=1.5in, left=1.5in, right=1.5in]{geometry}

\usepackage{tikz}

\allowdisplaybreaks

\theoremstyle{plain}
\newtheorem{theorem}{Theorem}[section]
\newtheorem{lemma}[theorem]{Lemma}
\newtheorem{cor}[theorem]{Corollary}

\newtheorem*{lemma*}{Lemma}
\newtheorem*{cor*}{Corollary}
\newtheorem*{theorem*}{Theorem}

\theoremstyle{definition}

\newtheorem{problem*}{Problem}

\theoremstyle{remark}

\newtheorem*{fact*}{Fact}
\newtheorem*{remark}{Remark}

\let\oldproofname=\proofname
\renewcommand{\proofname}{\rm\bf{\oldproofname}}

\newcommand{\R}{\mathbb R}

\newcommand{\Z}{\mathbb Z}

\makeatletter
\providecommand*{\twoheadrightarrowfill@}{%
  \arrowfill@\relbar\relbar\twoheadrightarrow
}
\providecommand*{\twoheadleftarrowfill@}{%
  \arrowfill@\twoheadleftarrow\relbar\relbar
}
\providecommand*{\xtwoheadrightarrow}[2][]{%
  \ext@arrow 0579\twoheadrightarrowfill@{#1}{#2}%
}
\providecommand*{\xtwoheadleftarrow}[2][]{%
  \ext@arrow 5097\twoheadleftarrowfill@{#1}{#2}%
}
\makeatother

\newcommand{\op}[1]{\operatorname{{#1}}}

\newcommand{\mc}[1]{\mathcal{{#1}}}

\newcommand{\Disc}{$\mc{D}$isc}

\newcommand{\GL}{\op{GL}}

\newcommand{\Diff}{\op{Diff}}

\begin{document}
    \title{Disc-presheaves and normal invariants}
    \author{Oliver H. Wang}
	\maketitle
\begin{abstract}
In this note, we elaborate on a result of Krannich--Kupers--Mehzer \cite{KKM} on \Disc-presheaves.
We show that equivalences of $2$-truncated \Disc-presheaves of manifolds yields homotopy equivalences with trivial normal invariant.
As an application, we show that two closed, simply connected manifolds are homeomorphic if and only if their $2$-truncated \Disc-presheaves are equivalent.
\end{abstract}
\section{Introduction}


We work in the $\infty$-categorical context so we write category to mean $\infty$-category, etc.
We let $\mc{S}$ denote the category of spaces, $\mc{S}_*$ denote the category of pointed spaces and $\mc{S}p$ denote the category of spectra.
Let $\mc{M}fd_d$ denote the category of smooth $d$-manifolds whose morphisms are smooth embeddings.
Let $\mc{D}isc_d$ denote the full subcategory whose objects are diffeomorphic to a disjoint union of $\R^d$'s.
Define $\mc{D}isc_{d,\le n}\subseteq\mc{D}isc_d$ to be the full subcategory whose objects are diffeomorphic to a disjoint union of at most $n$ copies of $\R^d$.
These may be regarded as quasi-categories by taking homotopy coherent nerves of simplicial categories.

A \emph{\Disc-presheaf} is a functor $(\mc{D}isc_d)^{op}\to \mc{S}$ and an \emph{$n$-truncated \Disc-presheaf} is a functor $\mc{D}isc_{d,\le n}^{op}\to\mc{S}$.
A smooth $d$-manifold $M$ determines a \Disc-presheaf $E_M$ via the restricted Yoneda embedding.
So we have $E_M(\underline{k}\times\R^d)=\op{Emb}(\underline{k}\times\R^d,M)$ where $\underline{k}:=\{1,\cdots,k\}$ and $\op{Emb}$ denotes the space of smooth embeddings in the smooth topology.
An important feature of \Disc-presheaves $X$ is that $X(\underline{k}\times\R^d)$ has an action of $\op{Diff}(\R^d)^{\times k}$.
The functors $E_M$ play a large role in embedding calculus and have been important in many recent results in high dimensional topology.
We refer to Kupers's ICM address \cite{KupersSurvey} for a thorough survey of \Disc-presheaves.

Let $X$ be an $n$-truncated $\mc{D}$isc-presheaf for $n\ge1$.
Then, there is an underlying space $|X|:=X(\R^d)/\Diff(\R^d)$.
When $X=E_M$, there are homotopy equivalences $|E_M|\simeq E_M(\R^d)/O(d)\simeq M$ where the first equivalence follows from the equivalence $\Diff(\R^d)\simeq O(d)$ and the second equivalence follows from restricting to the origin.
The underlying space has an associated $d$-dimensional vector bundle $\xi_X$ classified by by $X(\R^d)/\Diff(\R^d)\to */\Diff(\R^d)\simeq BO(d)$.
When $X=E_M$, this classifies the tangent bundle of $M$.
Similar statements hold in the topological setting.
If $f:X\to Y$ is a map of \Disc-presheaves, we write $|f|$ for the map of underlying spaces.
In the case $X=E_M$ and $Y=E_N$ for manifolds $M$ and $N$, we write $|f|$ for the map $M\to N$ obtained from the above homotopy equivalences instead.

One can also make analogous definitions using topological manifolds instead of smooth manifolds.
We will only work in the smooth setting to simplify notation.

Our main result is a generalization of \cite[Theorem 2.5]{KKM} (see also \cite[Theorem 5.1(i)]{KupersSurvey}).
Recall the Kervaire--Milnor sequence
\[
bP_{d+1}\to\Theta_d\to\op{coker}(J_d:\pi_d O\to \pi_d^s)
\]
where $\Theta_d$ is Kervaire--Milnor's group of homotopy spheres.
Krannich--Kupers--Mehzer show that, if the $2$-truncated \Disc-presheaves of two homotopy spheres $\Sigma_1$ and $\Sigma_2$ are equivalent, then $\Sigma_1\# (-\Sigma_2)\in bP_{d+1}$. 

In order to state our results, we first recall some notions from surgery theory.
The surgery theoretic structure set $\mc{S}^{h,\op{DIFF}}(N)$ of a smooth, connected $d$-manifold $N$ has elements represented by homotopy equivalences $g:M\to N$, where $M$ is a smooth manifold.
Two such elements are equivalent if they are smoothly $h$-cobordant and $\mc{S}^{h,\op{DIFF}}(N)$ is a pointed set with basepoint the identity.
When $d\ge5$ and $N$ is closed, the surgery exact sequence is the exact sequence of pointed sets
\[
\cdots\to L_{d+1}^h(\pi_1N)\to\mc{S}^{h,\op{DIFF}}(N)\to[N,F/O]\to L_d^h(\pi_1N)
\]
where $L_{d+1}^h(\pi_1N)$ is Wall's surgery group (technically, this group also depends on the orientation $\pi_1 N\to\Z/2$).
The map to $[N,F/O]$ sends a homotopy equivalence to its normal invariant, which we discuss in more detail in Section \ref{section : NI}.

When $N=S^d$, then $\mc{S}^{h,\op{DIFF}}(S^d)$ can be identified with Kervaire--Milnor's group of homotopy spheres $\Theta_d$.
Letting $e$ denote the trivial group,
\[
L_{d+1}^h(e)\cong\begin{cases}\Z&d+1\equiv0\text{ mod }4\\ \Z/2&d+1\equiv 2\text{ mod }4\\0&\text{otherwise}\end{cases}.
\]
This can be found in \cite[13.A]{Wall}, for instance.
The image of the $L$-group in $\Theta_d$ is precisely the subgroup $bP_{d+1}$ of homotopy spheres bounding parallelizable $(d+1)$-manifolds.
Recall $F:=\varinjlim_n hAut_*(S^n)$ is the space obtained by stabilizing the space of basepoint-preserving self-homotopy equivalences of $S^n$.
The space $F/O$ is defined as the fiber of $BO\to BF$ and $\pi_dF\cong\pi_d^s$ for $d\ge1$.
Under this isomorphism, the map $\pi_{d+1}BO\to\pi_{d+1}BF$ can be identified with $J_d$.
Hence, the cokernel of $J_d$ may be identified with a subgroup of $\pi_d F/O$.
The image of $\mc{S}^{h,\op{DIFF}}(S^d)\to\pi_dF/O$ is contained in $\op{coker}J_d$.
So, one recovers the Kervaire--Milnor sequence from the smooth surgery exact sequence.

There is also a topological surgery exact sequence which is algebraically nicer since the normal invariants can be identified with the generalized homology group $[N,F/TOP]\cong H_d(N;L(e)\langle1\rangle)$, where $L(e)\langle1\rangle$ is the $1$-connective surgery spectrum of the trivial group.
Thus the exact sequence becomes
\[
\cdots \to H_{d+1}(N;L(e)\langle1\rangle)\to L_{d+1}^h(\pi_1 N)\to\mc{S}^{h,\op{TOP}}(N)\to H_d(N;L(e)\langle1\rangle)\to \cdots
\]
and the map $H_{d+1}(N;L(e)\langle1\rangle)\to L_{d+1}^h(\pi_1 N)$ is the assembly map.

\subsection{Main Results}
Our main theorem is the following.
\begin{theorem}\label{thm: main}
Let $M$ and $N$ be closed, smooth manifolds and let $f:E_M\to E_N$ be an equivalence of $2$-truncated \Disc-presheaves.
Then $|f|:M\to N$ represents an element in the structure set $\mc{S}^{h,\op{DIFF}}(N)$ with trivial normal invariant.
The analogous result holds in the topological setting.
\end{theorem}

In the smooth setting, Theorem \ref{thm: main} shows that \Disc-presheaves can distinguish between certain smooth structures on manifolds.
The following corollary shows that \Disc-presheaves can actually distinguish between certain $PL$-structures.
\begin{cor}\label{cor: PL}
Let $M$ and $N$ be closed, smooth manifolds of dimension at least $5$.
Suppose there is an equivalence of $2$-truncated \Disc-presheaves $f:E_M\to E_N$ such that $|f|$ is homotopic to a homeomorphism and that $H_3(N;\Z)$ has no $2$-torsion.
Then, $M$ and $N$ are $PL$-homeomorphic.
\end{cor}
\begin{proof}
This follows from Theorem \ref{thm: main} and \cite[Theorem 19.2]{Rudyak}.
The underlying geometric fact is \cite[Theorem 17.3]{Rudyak}, which states that $F/PL\to F/TOP$ induces isomorphisms in all homotopy groups except $\pi_4$, where it is multiplication by $2$.
\end{proof}
As a concrete application of Corollary \ref{cor: PL}, one may consider exotic tori which are not $PL$-homeomorphic to the standard torus.
Such exotic tori exist in dimensions at least $5$ and, in dimension $5$ and $6$, they constitute all exotic tori by work of Kirby--Siebenmann \cite{KirbySiebenmann}.
When $d\ge 5$, any homotopy equivalence $M\to T^d$ from a closed $d$-manifold $M$ is homotopic to a homeomorphism by the Borel conjecture.
Thus Corollary \ref{cor: PL} implies that homotopy tori with equivalent \Disc-presheaves must be $PL$-homeomorphic.

\begin{remark}
We expect there to be a $PL$-analogue of \Disc-presheaves in which case Corollary \ref{cor: PL} should admit a strengthening where $M$ and $N$ need only be $PL$ manifolds.
\end{remark}

When $N$ is simply connected, the topological surgery exact sequence simplifies considerably.
The assembly map $H_*(N;L(e)\langle1\rangle)\to H_*(pt;L(e)\langle1\rangle)=L_*(e)$ splits and we obtain the following corollary in high dimensions.

\begin{cor}\label{cor: simply connected}
Let $M$ and $N$ be closed, simply connected topological manifolds.
Let $f:E_M\to E_N$ be an equivalence of $2$-truncated \Disc-presheaves.
Then $|f|:M\to N$ is homotopic to a homeomorphism.
\end{cor}

In dimensions at most $2$, the corollary is immediate.
In dimension $3$ it follows from the Poincar\'e conjecture and in dimension $4$ it follows from the fact that the trivial group is good, so topological surgery works as expected.

\begin{remark}
There are other instances where the assembly map splits.
Suppose $N$ is a closed, connected, aspherical topological manifold whose fundamental group satisfies the $K$ and $L$-theoretic Farrell--Jones conjecture.
Then the assembly map considered above is an isomorphism.
But these hypotheses imply the Borel conjecture for $N$ so the conclusion of Corollary \ref{cor: simply connected} does not give anything new.

Now let $N'$ be a closed, simply connected topological manifold and consider $N'\times N$, where $N$ is as above.
Then the assembly map factors as
\[
H_{d+1}(N'\times N;L(e)\langle1\rangle)\to H_{d+1}(N;L(e)\langle1\rangle)\cong L_{d+1}(\pi_1 N)
\]
which splits.
It follows that the conclusion of Corollary \ref{cor: simply connected} holds for $N'\times N$ in high dimensions.
Note that we are implicitly using the $K$-theoretic Farrell--Jones conjecture and the fact that $\pi_1 N$ is torsion-free to conclude that the Whitehead group $\op{Wh}(\pi_1 N)$ is trivial.
\end{remark}

\subsection{Acknowledgments}
The author was supported by the NSF under Award No. DMS-1839968.
The author would like to thank Alexander Kupers for helpful comments on a draft.

\subsection{AI Declaration}
AI was used only to search for references.

\section{Preliminaries}
\subsection{Normal Invariants}\label{section : NI}
In order to handle the normal invariants, it will be convenient to use the definitions in \cite[Section 7]{LuckMacko}.
Let $N$ be a smooth $d$-manifold.
One can associate to $N$ a Thom spectrum $\op{Th}(-TN)$ and a fundamental class of $N$ determines an element $c_N\in\pi_0\op{Th}(-TN)$.
This Thom spectrum is a shift of the Thom spectrum of the normal bundle $\nu_N$ of $N$, where $N$ is considered as a submanifold of Euclidean space.
If $f:M\to N$ is a homotopy equivalence then the fundamental class of $N$ determines a fundamental class of $M$ and thus an element $c_M\in\pi_0\op{Th}(-TM)$.
In this case, there is a map of Thom spectra $\op{Th}(-TM)\to\op{Th}(-TN)$ sending $c_M$ to $c_N$.
If $f$ is homotopic to a diffeomorphism, then there is an isomorphism $\nu_M\to f^*\nu_N$ of vector bundles over $M$ (possibly after stabilization) giving rise to this map of Thom spectra.
Following \cite[Definition 7.5]{LuckMacko}, we say that a homotopy equivalence $f:M\to N$ has \emph{trivial normal invariant} if there is an isomorphism of vector bundles $\nu_M\to f^*\nu_N$ inducing the equivalence of Thom spectra $\op{Th}(-TM)\to\op{Th}(-TN)$ such that $c_M$ is sent to $c_N$.

\subsection{Construction of Thom Spectra}
Theorem \cite[2.6]{KKM} shows that that the stable collapse map of a manifold can be recovered from the $2$-truncated \Disc-presheaf (both in the smooth setting and the topological setting).
Moreover, the authors show that the collapse map is natural in equivalences of \Disc-presheaves.

We recall the construction of the Thom spectrum $\op{Th}(-\xi_X)$ from \Disc-presheaves.
Let $X$ be an $n$-truncated \Disc-presheaf for $n\ge 2$.
The sphere bundle of $\xi_X$ is recovered from
\[
\frac{\op{Emb}(\R^d\sqcup\R^d,\R^d)}{\Diff(\R^d)^{\times 2}}\times_{\Diff(\R^d)}X(\R^d)
\]
and using that $\op{Emb}(\R^d\sqcup\R^d,\R^d)/\Diff(\R^d)^{\times 2}$ is $\GL_d(\R)$-equivariantly homotopy equivalent to $\R^d\setminus\{0\}$.
The projection to $|X|$ is obtained by mapping $\op{Emb}(\R^d\sqcup\R^d,\R^d)/\Diff(\R^d)^{\times 2}$ to a point.
It is clear from this construction that, if $f:X\to Y$ is an equivalence of $n$-truncated \Disc-presheaves, the induced map $S(\xi_X)\to S(\xi_Y)$ comes from a map of vector bundles.

Let $B$ denote $|X|$.
By unstraightening, the sphere bundle $S(\xi_X)$ is equivalent to a functor $B\to\mc{S}$ sending $0$-simplices to the fiber of the bundle.
The parametrized Thom spectrum $\op{Th}_B(-\xi_X)$ may be regarded as a functor from $B\to\mc{S}p$ given by the composition
\[
B\xrightarrow{S(\xi_B)}\mc{S}\xrightarrow{\Sigma}\mc{S}_*\xrightarrow{\Sigma^\infty}\mc{S}p\xrightarrow{D}\mc{S}p^{op}
\]
where $D$ is Spanier--Whitehead duality.
The functor $\Sigma$ sends a space to its suspension, where the basepoint of the suspension is one of the cone points.
Taking the colimit over $B$ gives the Thom spectrum $\op{Th}(-\xi_X)$.

\section{Proof of Theorem \ref{thm: main}}
The action of $O(d)$ on $S^{d-1}$ given by restricting the action on $\R^{d}$ to the unit sphere gives a map $O(d)\to Map_{\mc{S}}(S^{d-1},S^{d-1})$ and thus a functor $BO(d)\to\mc{S}$.
Since $S(\xi_X)$ is the sphere bundle of a vector bundle on $X$, there is a factorization $B\to BO(d)\to\mc{S}$.
Similarly, the action of $O(d)$ on $S^d$ given by taking the one-point compactification of $\R^d$ gives a map $O(d)\to Map_{\mc{S}_*}(S^d,S^d)$.
If $X$ and $Y$ are $2$-truncated \Disc-presheaves, an equivalence induces an equivalence of vector bundles $\xi_X\cong\xi_Y$.
This induces an equivalence of sphere bundles $S(\xi_X)\cong S(\xi_Y)$.
By unstraightening, we obtain the following diagram.
\[
\begin{tikzpicture}[scale=1.5]
\node (A) at (0,0) {$B\times\Delta^1$};\node (B) at (2,0) {$\mc{S}$};\node (C) at (4,0) {$\mc{S}_*$};\node (D) at (6,0) {$\mc{S}p$};\node (E) at (8,0) {$\mc{S}p^{op}$};
\node (F) at (2,1) {$BO(d)$};\node (G) at (4,1) {$BO(d)$};\node (H) at (6,1) {$BO$};\node (I) at (8,1) {$BO$};
\path[->] (A) edge (B) (B) edge node[above]{$\Sigma$} (C) (C) edge node[above]{$\Sigma^{\infty}$} (D) (D) edge node[above]{$D$} (E) (A) edge (F) (F) edge node[above]{$=$} (G) (G) edge (H) (H) edge[dashed] (I) (F) edge (B) (G) edge (C) (H) edge (D) (I) edge (E);
\end{tikzpicture}.
\]
To show the map of Thom spectra is induced by a map of vector bundles, it suffices to show that there is a map $BO\to BO$ such that the right square commutes.
This is the content of the next lemma.

\begin{lemma}\label{lem: duality preserves O}
Let $t:O\to O$ denote the transpose.
The following diagram commutes up to homotopy
\[
\begin{tikzpicture}[scale=1.5]
\node (A) at (0,1) {$O$};\node (B) at (2,1) {$Map_{\mc{S}p}(S^0,S^0)$};
\node (C) at (0,0) {$O$};\node (D) at (2,0) {$Map_{\mc{S}p}(S^0,S^0)$};
\path[->] (A) edge (B) (A) edge node[left]{$t$} (C) (B) edge node[left]{$D$} (D) (C) edge (D);
\end{tikzpicture}
\]
\end{lemma}
\begin{proof}
For a spectrum $X$, let $e:X\wedge DX\to S^0$ denote the evaluation.
If $f:S^0\to S^0$ is a map of spectra, then $Df$ is determined by the property
\[
e(f\wedge\op{id}_{DX})=e(\op{id}_X\wedge Df):X\wedge DX\to S^0.
\]

Suppose $A\in O(n)$ is an orthogonal transformation.
Then, fixing an identification $DS^0$ with $S^0$, the map $e(A\wedge\op{Id}_{DS^0}):S^0\wedge S^0\to S^0\wedge S^0$ can be modeled, using suspension spectra, by the map $S^{2n}\to S^{2n}$ arising from the one point compactification of the linear transformation $\begin{pmatrix}A&0\\0&\op{id}_{\R^n}\end{pmatrix}$.
By conjugating with a rotation matrix, we obtain a homotopy making the following diagram commute
\[
\begin{tikzpicture}[scale=1.5]
\node (A) at (0,1) {$O(n)$};\node (B) at (2,1) {$Map_{\mc{S}_*}(S^n,S^n)$};
\node (C) at (0,0) {$O(n)$};\node (D) at (2,0) {$Map_{\mc{S}_*}(S^n,S^n)$};
\path[->] (A) edge (B) (A) edge node[left]{$t$} (C) (B) edge node[left]{$D$} (D) (C) edge (D);
\end{tikzpicture}.
\]
The lemma follows from stabilizing this diagram.
\end{proof}

The proof of Lemma \ref{lem: duality preserves O} also shows that the analogous result for homeomorphisms is true.

\begin{lemma}\label{lem: duality preserves TOP}
There is a map $t:TOP\to TOP$ such that the following diagram commutes up to homotopy
\[
\begin{tikzpicture}[scale=1.5]
\node (A) at (0,1) {$TOP$};\node (B) at (2,1) {$Map_{\mc{S}p}(S^0,S^0)$};
\node (C) at (0,0) {$TOP$};\node (D) at (2,0) {$Map_{\mc{S}p}(S^0,S^0)$};
\path[->] (A) edge (B) (A) edge node[left]{$t$} (C) (B) edge node[left]{$D$} (D) (C) edge (D);
\end{tikzpicture}.
\]
\end{lemma}

\bibliographystyle{alpha}
\bibliography{discpresheaves}
\end{document}